\documentclass[11pt]{amsart}

\usepackage[margin=1.2in]{geometry}
\usepackage{amsmath,amssymb,amsthm,url} 
\usepackage{graphicx}

\title{On integers $n$ for which $X^n-1$ has a divisor\\ of every degree}
\renewcommand\rightmark{\centerline{\normalsize\sc
On integers $n$ for which $X^n-1$ has a divisor of every degree}\thepage}

\author{Carl Pomerance}
\address{Department of Mathematics\\Dartmouth College\\Hanover, NH 03755}
\email{carlp@math.dartmouth.edu}

\author{Lola Thompson}
\address{Department of Mathematics\\ Oberlin College\\ Oberlin, OH 44074}
\email{lola.thompson@oberlin.edu}

\author{Andreas Weingartner}
\address{Department of Mathematics\\Southern Utah University\\Cedar City, UT 84720}
\email{weingartner@suu.edu}

\DeclareMathAlphabet{\curly}{U}{rsfs}{m}{n}

\newtheorem{thm}{Theorem}[section]

\newtheorem{lemma}[thm]{Lemma}
\newtheorem{cor}[thm]{Corollary}

\newtheorem*{definition}{Definition}

\newtheorem*{theorem*}{Theorem}
\theoremstyle{remark}

\numberwithin{equation}{section}

\newcommand\Z{\mathbf{Z}}

\newcommand\rad{\mathrm{rad}}

\renewcommand{\phi}{\varphi}
\makeatletter
\renewcommand{\pod}[1]{\mathchoice
  {\allowbreak \if@display \mkern 18mu\else \mkern 8mu\fi (#1)}
  {\allowbreak \if@display \mkern 18mu\else \mkern 8mu\fi (#1)}
  {\mkern4mu(#1)}
  {\mkern4mu(#1)}
}

\begin{document}

\begin{abstract} 
A positive integer $n$ is called $\varphi$-practical if the polynomial $X^n-1$ has a divisor in $\Z[X]$ of every 
degree up to $n$. In this paper, we show that the count of $\varphi$-practical numbers in $[1, x]$ is asymptotic to $C x/\log x$ for some positive constant $C$ as $x \rightarrow \infty$. 
\end{abstract}


\maketitle

\section{Introduction}

Let $n$ be a positive integer. Following Srinivasan \cite{Srini}, 
we say that $n$ is \textit{practical} if every natural number up to $n$ can be written as a subsum 
of the natural divisors of $n$. The practical numbers have been well-studied beginning with Erd\H{o}s, 
who stated in a 1948 paper \cite{erdos2} that the practical numbers have asymptotic density $0$. Over the 
next half-century, various authors worked in pursuit of a precise estimate for the count of practical numbers in the interval $[1, x].$ 
Until recently, the strongest result in this vein was a pair of Chebyshev-type inequalities due to Saias \cite{saias1}:

\begin{thm}[Saias, 1997] 
\label{thm-saias}
Let $P\kern-1pt R(x)$ denote the number of practical numbers in $[1, x]$. 
There exist positive constants $\kappa_1$ and $\kappa_2$ such that for all $x\ge2$,
$$\kappa_1\frac{x}{\log x} \leq P\kern-1pt R(x) \leq \kappa_2\frac{x}{\log x}.$$ 
\end{thm}

It was conjectured in 1991 by Margenstern \cite{margenstern} that 
$P\kern-1pt R(x) \sim \kappa \frac{x}{\log x}$ as $x \rightarrow \infty$ 
for some positive constant $\kappa$. 
Such an asymptotic for $P\kern-1pt R(x)$ was finally obtained by the third author \cite{weingartner}, 
resolving Margenstern's conjecture affirmatively. 
 
\begin{thm}[Weingartner, 2015] 
\label{thm-wein}
There is a positive constant $\kappa$ such that for $x \geq 3$ 
$$P\kern-1pt R(x) = \frac{\kappa x}{\log x}\left\{1 + O\left(\frac{\log \log x}{\log x}\right)\right\}.$$ 
\end{thm}

A key property of practical numbers used in these results had been proved in the 1950s
by Stewart \cite{stew} and Sierpi\'nski \cite{sier}, who gave a recursive characterization:
\textit{The number $1$ is practical and if $n$ is practical
and $p$ is a prime, then $pn$ is practical if and only if $p\le 1+\sigma(n)$.  All practical numbers
arise in this way.}  Here $\sigma$ is
the sum-of-divisors function.

Perhaps more centrally placed in the anatomy of integers are the 2-dense numbers:
A positive integer $n$ is \textit{$2$-dense} if each interval $[y,2y]$ contained in $[1,n]$ has
a divisor of $n$.  The recursive criterion for a number to be 2-dense: \textit{The number $1$ is
$2$-dense and if $n$ is $2$-dense and $p$ is a prime, then $pn$ is $2$-dense if and only if
$p\le 2n$.  All $2$-dense numbers arise in this way.}  
Analogues of Theorems \ref{thm-saias}, \ref{thm-wein} hold as well for
2-dense numbers, and by essentially the same proofs.  (For 2-dense numbers
the analogue of Theorem \ref{thm-wein} has the slightly stronger error
term $O(1/\log x)$.)

This paper discusses the related concept of \textit{$\varphi$-practical} numbers:  A positive integer $n$
is $\varphi$-practical if $X^n-1$ has divisors in ${\bf Z}[X]$ of every degree to $n$.  Since $X^n-1$
is squarefree with its irreducible factors having degrees $\varphi(d)$ as $d$ runs over the divisors of $n$
(where $\varphi$ is Euler's function), it follows that $n$ is $\varphi$-practical if and only if
each natural number to $n$ is a subsum of the set $\{\varphi(d):d\mid n\}$.
These numbers were first considered by the second author in her Ph.D.\ thesis.
It is natural to consider whether the methods for practical numbers and 2-dense numbers can be
used for $\varphi$-practical numbers.

Complicating things is that there is no simple growth condition on the prime factors that
categorizes the $\varphi$-practical numbers.  However, there are some conditions that come close
to doing this, see \cite{thompson}:
\begin{itemize}
\item  If $n$ is $\varphi$-practical and $p$ is a prime that does not divide $n$, then $pn$ is
$\varphi$-practical if and only if $p\le n+2$.
\item If $n$ is $\varphi$-practical and $p$ is a prime that does not divide $n$, then $p^jn$
is $\varphi$-practical for each integer $j\ge 2$ if and only if $p\le n+1$.
\end{itemize}
Consider the set ${\mathcal W}$ built up recursively by the rules $1\in{\mathcal W}$ and if
$n\in{\mathcal W}$ and $p$ is prime, then $pn\in{\mathcal W}$ if and only if $p\le n+2$.
We say a member of ${\mathcal W}$ is \emph{weakly $\varphi$-practical}.
As shown in \cite{thompson}, every $\varphi$-practical number is weakly $\varphi$-practical.
As the second bullet above indicates, not all weakly $\varphi$-practical numbers are $\varphi$-practical.
With practical and 2-dense numbers, if the largest prime factor is removed, one again has
a practical or 2-dense number, respectively. The same holds for weakly $\varphi$-practical
numbers. However, this is not the case for $\varphi$-practicals.
In particular, there are $\varphi$-practical numbers $pn$ where $p$ is
greater than all of the primes dividing $n$, but $n$ itself is not $\varphi$-practical. An
example is $pn=315=3^2\cdot5\cdot7$.

It was noted in \cite{thompson} that every even number that is weakly $\varphi$-practical is
$\varphi$-practical.  Using this, it follows from \cite{weingartner}
that the analogue of Theorem \ref{thm-wein} holds for even $\varphi$-practical numbers.

Meanwhile, in \cite{thompson}, the second author was able to show the analogue of Theorem \ref{thm-saias}
for all of the $\varphi$-practical numbers.  Let $P_\varphi(x)$ denote the number of $\varphi$-practical
numbers in $[1,x]$.
\begin{thm}[Thompson, 2013]
\label{thm-thompson}
There are positive numbers $\kappa_3,\kappa_4$ such that for all $x\ge 2$,
$$
\kappa_3 \frac{x}{\log x} \leq P_\varphi(x) \leq \kappa_4 \frac{x}{\log x}.
$$ 
\end{thm}
In the present paper, we obtain an asymptotic for the count of $\varphi$-practical numbers up to $x$. Our main theorem can be stated as follows.

\begin{thm}
\label{thm-main}
There is a positive number $C$ such that for $x \geq 2$ 
$$P_\varphi(x) = \frac{C x}{\log x}\left\{1 + O\left(\frac{1}{\log x}\right)\right\}.$$ 
\end{thm}

Our strategy is to try to use the \textit{squarefree-squarefull} decomposition of a
positive integer $n$, namely $n=qs$ where $s$ is the largest squarefull divisor of $n$
(a number is \textit{squarefull} if it is divisible by the square of each of its
prime factors).  The idea is to fix the squarefull part $s$ and obtain an asymptotic
for the $\varphi$-practicals with this squarefull part.  The plan works in a
fairly straightforward way for some cases, like $s=1$ and $s=4$, but it is not so
easy to do for other cases, such as $s=9$.

Our methods do not yield an explicit estimate for the constant $C$ that appears in the statement of Theorem \ref{thm-main}. The numerical computations in the second author's Ph.D. thesis (summarized here in Table \ref{phi-ratios},
with a new calaculation at $10^{10}$) 
seem to suggest that $C \approx 1$.  In the final section we give an argument for why $C$ may be slightly
less than 1. 
 
\begin{table}[h!]\label{phi-ratios}
\begin{center}
    \begin{tabular}{ | l | l | c |}
    \hline
    $X$ & $P_{\varphi}(X)$ & $P_{\varphi}(X)/(X/\log X)$\\ \hline
    $10^1$ & 6 & 1.381551 \\
    $10^2$ & 28 & 1.289448 \\
    $10^3$ & 174 & 1.201949 \\
    $10^4$ & 1198 & 1.103399 \\
    $10^5$ & 9301 & 1.070817 \\
    $10^6$ & 74461 & 1.028717 \\
    $10^7$ & 635528 & 1.024350 \\
    $10^8$ & 5525973 & 1.017922 \\
    $10^9$ & 48386047 & 1.002717 \\
    $10^{10}$ & 431320394 & 0.993152 \\
    \hline
    \end{tabular}
\end{center}
\caption{Ratios for $\varphi$-practicals}\label{table:fphi}
\end{table}

\section{Preliminaries}\label{sec-preliminaries}

In this section, we set the notation and define some terminology that will be used throughout the paper.  We also establish some lemmas on the
distribution of squarefree numbers without small prime factors.

We use the letter $p$, with or without subscripts, to denote primes. 

For an integer $n>1$, let $P^+(n)$ denote the largest prime dividing $n$, and let $P^-(n)$ denote
the smallest prime dividing $n$.  Further, we let $P^+(1)=1$ and $P^-(1)=+\infty$.

We say that $d$ is an {\em initial divisor} of $n$ if $d\mid n$ and $P^+(d)<P^-(n/d)$.  

As mentioned earlier, a positive integer $n$ is {\em squarefull} if $p^2\mid n$ for each prime $p\mid n$.
The {\em squarefull part} of $n$ is the largest squarefull divisor of $n$.


We write $A(x)\ll B(x)$ if $A(x)=O(B(x))$.  We write $A(x)\asymp B(x)$ if
$A(x)\ll B(x)\ll A(x)$.

For $u \geq 1$, we define Buchstab's function $\omega(u)$ to be the unique continuous solution to the equation $$(u \omega(u))' = \omega(u-1) \hspace{0.2 in} (u > 2)$$ with initial condition 
$$u\omega(u) = 1 \hspace{0.2 in} (1 \leq u \leq 2).$$ 
For $u < 1$, let $\omega(u) = 0$. 
We have
\begin{equation}
\label{eq:omega}
|\omega(u)-e^{-\gamma}|\le1/\Gamma(u+1),~~u\ge0,
\end{equation}
see \cite[Lemma 2.1]{weingartner}.  Let 
$$\Phi(x, y) = \sum_{\substack{n \leq x \\ P^-(n) > y}} 1.$$ 
We record the following result in \cite[Lemma 2.2]{weingartner}:
For $x\ge1, y\ge2, u:=\log x/\log y$, we have
\begin{equation}
\label{eq:Phi}
\Phi(x,y)=e^\gamma x\omega(u)\prod_{p\le y}\left(1-\frac1p\right)+O\left(\frac y{\log y}+\frac{xe^{-u/3}}{(\log y)^2}\right).
\end{equation}

We will need a variant of $\Phi(x,y)$ for squarefree numbers.
\begin{definition} For a positive integer $n$, let 
$$
\Phi_0(x, y) := \displaystyle\sum_{\substack{n \leq x \\ P^-(n) > y}} \mu^2(n).
$$ 
\end{definition}
In other words, $\Phi_0(x, y)$ detects the squarefree values of $n$ counted in $\Phi(x, y)$. 
The following two lemmas allow us to estimate this function.

\begin{lemma}
\label{phinot} 
For $x \geq 1$ and $y \geq 2$, we have $$\Phi_0(x, y) = \Phi(x, y) + O\left(\frac{x}{y \log y}\right).$$ 
\end{lemma}

\begin{proof} Observe that 
\begin{align*} 
0\le\Phi(x, y) - \Phi_0(x, y) 
\leq \sum_{p > y} \sum_{\substack{n \leq x \\ p^2 \mid n}} 1  
\leq \sum_{p > y} \left\lfloor \frac{x}{p^2} \right\rfloor  
\leq x \sum_{p > y} \frac{1}{p^2}  = O\left(\frac{x}{y \log y}\right).
\end{align*} 
\end{proof}

\begin{lemma}
\label{lemma2'} 
 For $x \geq 1$ and $2\le y \le e^{2\sqrt{\log x}}$, we have 
$$
\Phi_0(x, y) = \frac{6}{\pi^2} x \prod_{p \le y} \left(1 + \frac{1}{p}\right)^{-1} + O\left(\frac{x}{e^{\frac16\sqrt{\log x}}}\right).
$$ 
\end{lemma}
\begin{proof} 
By definition of $\Phi_0(x, y)$, we have 
$$
\Phi_0(x, y)  = \sum_{\substack{n \leq x \\ P^-(n) > y}} \mu^2(n) 
 = \sum_{\substack{n \leq x \\ P^-(n) > y}} \sum_{d^2 \mid n} \mu(d) 
 = \sum_{\substack{d \leq \sqrt{x}  \\
P^-(d) > y}} \mu(d) \Phi(x/d^2, y).
$$
We split the values of $d$ into two ranges: $d > e^{(\log x)^{1/2}}$ and $d \leq e^{(\log x)^{1/2}}$. 
Since $\Phi(x/d^2,y)\le x/d^2$,
the contribution from the terms where $d > e^{(\log x)^{1/2}}$ is trivially $O\left(\frac{x}{e^{(\log x)^{1/2}}}\right)$.
 For the remainder of the proof, we consider only those $d$ for which $d \leq e^{(\log x)^{1/2}}$. 
 From \eqref{eq:Phi}, we have 
\begin{align}
\label{weingartnerl2} 
\Phi(x/d^2, y) = e^\gamma \omega\left(\frac{\log(x/d^2)}{\log y}\right) \frac{x}{d^2} \prod_{p \leq y} \left(1 - \frac{1}{p}\right) 
+ O\left(\frac y{\log y}+\frac{x}{d^2 e^{u'/3}}\right),
\end{align} 
where $u' = \frac{\log(x/d^2)}{\log y}$.  Let $u=\frac{\log x}{\log y}$,
so that $u'= u(1+o(1))$ as $x\to \infty$.
Thus, for values of $d$ in this range and using \eqref{eq:omega}, 
we have $\omega\left(\frac{\log (x/d^2)}{\log y}\right) = e^{-\gamma} + O(\frac{1}{e^u})$.  
Also, $d^2e^{u'/3}\gg d^{4/3}e^{u/3}$. 
Inserting these estimates into \eqref{weingartnerl2} yields 
$$
\Phi(x/d^2, y) = \frac{x}{d^2} \prod_{p \leq y} \left(1 - \frac{1}{p}\right) + O\left(\frac{x}{d^{4/3} e^{u/3}}\right).$$ 
Therefore, for $y \leq e^{2\sqrt{\log x}}$, we have 
\begin{align}
\label{phi0} 
\Phi_0(x, y) = x\sum_{\substack{d \leq e^{(\log x)^{1/2}}\\ P^-(d) > y}} 
\frac{\mu(d)}{d^2} \prod_{p \leq y} \left(1 - \frac{1}{p}\right) 
+ O\left(\frac{x}{e^{\frac16(\log x)^{1/2}}}\right).
\end{align} 
We can rewrite the sum over $d$ as 
\begin{align}
\label{diffofsums} 
\sum_{P^-(d) > y} \frac{\mu(d)}{d^2} - \sum_{\substack{d > e^{(\log x)^{1/2}} \\ P^-(d) > y}} \frac{\mu(d)}{d^2}.
\end{align} 
As above, the contribution from the subtracted sum is $O\left(\frac{1}{e^{(\log x)^{1/2}}}\right)$. 
Using this in \eqref{phi0} we have 
\begin{align*}
\Phi_0(x,y)&=x\prod_{p\le y}\left(1-\frac1p\right)
\sum_{P^-(d)>y}\frac{\mu(d)}{d^2}
+O\left(\frac{x}{e^{\frac16(\log x)^{1/2}}}\right)\\
&= 
x\prod_{p \le y} \left(1 - \frac{1}{p}\right) \prod_{p > y} \left(1 - \frac{1}{p^2}\right) 
+ O\left(\frac{x}{e^{\frac16(\log x)^{1/2}}}\right).
\end{align*} 
The products can be rewritten as 
$$
\prod_{p \le y} \left(1 - \frac{1}{p} \right) \frac{\displaystyle\prod_{p} 
\left(1 - \frac{1}{p^2}\right)}{\displaystyle\prod_{p \le y} 
\left(1 - \frac{1}{p^2}\right)} = \frac6{\pi^2}\prod_{p \le y} \left(1 + \frac{1}{p}\right)^{-1},
$$ 
yielding our result. 
\end{proof}


\section{Growing ``squarefreely"}
\label{sec-freely}

The comments in the introduction about $\varphi$-practical numbers indicate that the situation
is simpler for squarefree ones.  In particular, a squarefree number is $\varphi$-practical if and
only if its canonical prime factorization $p_1\dots p_k$, where $p_1<\dots<p_k$, has each
$p_i\le 2+p_1\dots p_{i-1}$.  In this section we shall obtain an asymptotic estimate for the
distribution of squarefree $\varphi$-practical numbers, doing so in a somewhat more general setting.

Let $\theta$ be any real-valued arithmetic function defined on $[1,\infty)$ with $\theta(1)\ge2$ and
$x\le\theta(x)\ll x$.  Let $m$ be an arbitrary positive integer.
Let $\mathcal{B}_m$ denote the set of positive integers $mb$, where $b$ is squarefree,
$P^+(m)<P^-(b)$, and the canonical prime factorization of $b=p_1\dots p_k$, with $p_1<\dots <p_k$,
satisfies $p_i\le\theta(mp_1\dots p_{i-1})$ for each $i=1,\dots,k$.
Let $B_m(x) = \#\{n \leq x : n \in \mathcal{B}_m\}$. 
Observe that if $m=1$ and if $\theta(n) = n + 2$, then $B_1(x)$ counts the number of 
squarefree $\varphi$-practical numbers $n \leq x$.   Also observe that if $\theta(n)=n+2$ and
$m$ is $\varphi$-practical,
so too is every member of $\mathcal{B}_m$.  
Let $\chi_m(n)$ denote the characteristic function of $\mathcal{B}_m$, i.e., 
\begin{displaymath}
   \chi_m(n) = \left\{
     \begin{array}{lr}
       1 & \mathrm{if} \ n \in \mathcal{B}_m\\
       0 & \mathrm{otherwise}.
     \end{array}
   \right.
\end{displaymath} 

\begin{thm}
\label{thm-starterasymp}
Let $r_m = m^{-1} (\log 2m)^6$. 
There is a sequence of real numbers $c_m$ such that 
$$B_m(x)= c_m \frac{x}{\log x} + O\left(r_m \frac{x }{\log^2 x}\right) \qquad (m\ge 1, \ x\ge 2).$$
\end{thm}

Our proof will follow closely the proof of \cite[Theorem 1.2]{weingartner}. We begin with a simple upper bound for $B_m(x)$.

\begin{lemma}
\label{Bm-upper-bound}
For $m\ge1$ , $x \geq 1$, we have
$$ B_m(x) \ll \frac{x \log 2m}{m \log2x}.$$ 
\end{lemma}

\begin{proof}
If $mb \in \mathcal{B}_m$ and $mb\le x$, then 
$b=p_1\dots p_k \le x/m$ and $p_i\le Cm p_1\dots p_{i-1}$ for each $i=1,\dots,k$, for some positive constant $C$, since $\theta(n) \ll n$. 
Theorem 1 of \cite{saias1} implies that 
$$ B_m(x) \ll \frac{x \log Cm}{m \log C x} \ll \frac{x \log 2m}{m \log2x}.$$
\end{proof}

The following lemma is the analogue of \cite[Lemma 5.2]{weingartner}.

\begin{lemma}\label{lem-phi0starter} For $m\ge1$ , $x \geq 1$,
we have 
\begin{equation}
\label{eq:bm1}
B_m(x)=
\Phi_0(x/m, P^+(m)) - \sum_{mb \leq \sqrt{x}} \chi_m(mb) \Phi_0(x/mb, \theta(mb)) + B_m(\sqrt{x}) 
\end{equation}
\end{lemma}

\begin{proof} 
We may assume $x\ge m$, or else each term in \eqref{eq:bm1} vanishes.
For all squarefree $j \leq x/m$ with $P^-(j) > P^+(m)$, we can decompose $j = bk$, where 
$mb \in \mathcal{B}_m$ and $P^-(k) > \theta(mb).$ As in \cite{weingartner}, this decomposition is unique. 
Moreover, since $P^-(j) > P^+(m)$, we necessarily have $P^-(b) > P^+(m)$. Thus, 
$$
\Phi_0(x/m, P^+(m)) = \sum_{b \leq x/m } \chi_m(mb) 
\sum_{\substack{k \leq x/mb \\ P^-(k) > \theta(mb)}} \mu^2(k) 
= \sum_{b \leq x/m } \chi_m(mb) \Phi_0(x/mb, \theta(mb)).
$$
Since $\theta(mb) \geq mb$, we have $\Phi_0(x/mb, \theta(mb)) = 1$ for $\sqrt{x} < mb \leq x.$ 
As a result, 
$$
\Phi_0(x/m,P^+(m))=B_m(x) - B_m(\sqrt{x}) + \sum_{mb \leq \sqrt{x} } \chi_m(mb) \Phi_0(x/mb, \theta(mb)).
$$
Solving for $B_m(x)$ yields the result.
\end{proof}

Next, we prove a variant of \cite[Lemma 5.3]{weingartner} for $\Phi_0(x, y).$
\begin{lemma}\label{lem-B(x)starter1} 
For $m\ge 1$, $x\ge 1$ we have
\begin{align*} 
B_m(x) &= \frac{6}{\pi^2} \frac{x}{m} \prod_{p \leq P^+(m)} \left(1 + \frac{1}{p}\right)^{-1} 
- \frac{6}{\pi^2} x \sum_{mb \leq e^{\sqrt{\log x}} } \frac{\chi_m(mb)}{mb} 
\prod_{p \leq \theta(mb)} \left(1 + \frac{1}{p}\right)^{-1} \\ 
&- x \sum_{e^{\sqrt{\log x}} < mb \leq \sqrt{x} } \frac{\chi_m(mb)}{mb} 
e^\gamma \omega \left(\frac{\log(x/mb)}{\log \theta(mb)} \right) 
\prod_{p \leq \theta(mb)} \left(1 - \frac{1}{p}\right) + O\left( \frac{r_m x}{\log^22x} \right).
\end{align*} 
\end{lemma}

\begin{proof}
Since $B_m(x) \le x/m$, and each of the three main terms in Lemma \ref{lem-B(x)starter1} is trivially $\ll x m^{-1} \log 2x$, 
everything is absorbed by the error term as long as $\log 2x \ll \log^2 2m$.
Hence we may assume $\log 2x > C \log^2 2m$, for some sufficiently large constant $C$, as we estimate
the right-hand side of \eqref{eq:bm1}. In particular, we may assume hereafter that $x\ge 2$. 

If $m>1$, Lemma \ref{lemma2'} implies that
$$
\Phi_0(x/m, P^+(m)) = \frac{6}{\pi^2} \frac{x}{m} \prod_{p \leq P^+(m)} \left(1 + \frac{1}{p}\right)^{-1} 
+ O\left(\frac{x}{m \log^2 x}\right).
$$ 
This also holds for $m=1$ using a standard result on the distribution of squarefree numbers.

Now suppose that $mb \leq e^{\sqrt{\log x}}$. 
Since $2\le \theta(mb)\ll mb $, we can estimate $ \Phi_0(x/mb, \theta(mb))$ by Lemma \ref{lemma2'}. 
The cumulative error in \eqref{eq:bm1}
from $mb \leq e^{\sqrt{\log x}}$ is $\ll x m^{-1} / \log^2 x$.
 
Using Lemma \ref{phinot} and \eqref{eq:Phi}, in the range $e^{\sqrt{\log x}} < mb \leq \sqrt{x}$ we have
\begin{align*} 
 \chi_m(mb)& \Phi_0(x/mb, \theta(mb)) = \\ &x \chi_m(mb) \frac{e^\gamma}{mb} 
\omega\left(\frac{\log(x/mb)}{\log \theta(mb)}\right) \prod_{p \leq \theta(mb)} 
\left(1 - \frac{1}{p}\right) + O(E_1) + O(E_2) + O(E_3),
\end{align*} 
where 
$$
E_1=
\frac{x\chi_m(mb)}{mb\, \theta(mb)\log\theta(mb)},~~
E_2=\frac{x\chi_m(mb)e^{-\log(x/mb)/3\log\theta(mb)}}{mb(\log\theta(mb))^2},~~
E_3=\frac{\chi_m(mb)\theta(mb)}{\log\theta(mb)}.
$$

We wish to show that the sum of each $E_j$ for $e^{\sqrt{\log x}}<mb\le\sqrt{x}$ is $O(\frac{x \log 2m}{m \log^2 x})$.  
Since $mb\le\theta(mb)\ll mb$, this is immediate for $E_1$.  
The sum of $E_3$ is $\ll B_m(\sqrt{x}) \sqrt{x}/\log x$, which is acceptable by Lemma \ref{Bm-upper-bound}.
The argument for $E_2$ is a bit more delicate.  By Lemma \ref{Bm-upper-bound},
the sum of $\chi_m(mb)/mb$ for $mb$ in a dyadic interval $[2^j,2^{j+1})$ is at most
$B_m(2^{j+1})/2^j\ll (\log 2m)/(mj)$.
The contribution to $E_2$ from this interval is
$\ll (\log 2m) m^{-1} xe^{-c(\log x)/j}/j^3$ for some positive constant $c$.  Summing this over the larger
range $j\ge1$ gets an estimate of $O(\frac{x \log 2m}{m \log^2 x})$.

Finally, we replace the term $B_m(\sqrt{x})$ in \eqref{eq:bm1} with $O( \sqrt{x}/m)$.

Inserting each of these estimates into \eqref{eq:bm1}, and observing that each of the error terms is 
$O(\frac{x \log 2m}{m \log^2 x})$, produces our desired result.
\end{proof}

Our version of \cite[Lemma 5.4]{weingartner} is as follows.

\begin{lemma}\label{lem-B(x)simplification} 
For $m\ge 1$ we have
$$
\frac{1}{m} \prod_{p \leq P^+(m)} \left(1 + \frac{1}{p}\right)^{-1} =
 \sum_{b \geq 1 } \frac{\chi_m(mb)}{mb} \prod_{p \leq \theta(mb)} \left(1 + \frac{1}{p}\right)^{-1}.$$  
\end{lemma}

\begin{proof}
Let $m\ge 1$ be arbitrary but fixed. 
Dividing each term by $x$ in Lemma \ref{lem-B(x)starter1} and rearranging, 
we obtain, by Lemma \ref{Bm-upper-bound}, 
\begin{align*} 
\frac{6}{\pi^2} \frac{1}{m} \prod_{p \leq P^+(m)} \left(1 + \frac{1}{p}\right)^{-1} 
&= \frac{6}{\pi^2} \sum_{mb \leq e^{\sqrt{\log x}} } \frac{\chi_m(mb)}{mb} \prod_{p \leq \theta(mb)} 
\left(1 + \frac{1}{p} \right)^{-1} \\ 
&+ \sum_{e^{\sqrt{\log x}} \leq mb \leq \sqrt{x} } \frac{\chi_m(mb)}{mb} e^\gamma 
\omega\left(\frac{\log(x/mb)}{\log \theta(mb)}\right) \prod_{p \leq \theta(mb)} \left(1 - \frac{1}{p}\right) + o(1),
\end{align*} 
as $x \rightarrow \infty$. Since $e^\gamma \omega(\log(x/mb)/\log \theta(mb)) \ll 1$, and 
$\prod_{p \leq \theta(mb)}(1 - 1/p) \ll 1/\log mb$, it follows via 
partial summation and Lemma \ref{Bm-upper-bound}  that the second sum is $o(1)$. 
We can extend the first sum to include all $mb \geq 1$, since the contribution 
from those $mb$ with $mb > e^{\sqrt{\log x}}$ is $o(1)$. 
Dividing both sides of the equation by $6/\pi^2$ and letting $x \rightarrow \infty$ completes the proof. 
\end{proof}

Next, we prove a version of \cite[Lemma 5.5]{weingartner}.

\begin{lemma} 
\label{lem-5.5}
For $m\ge 1$, $x \geq 1$ we have 
$$
B_m(x) = x \sum_{b \geq 1 } \frac{\chi_m(mb)}{mb \log \theta(mb)} 
\left(e^{-\gamma} - \omega\left(\frac{\log(x/mb)}{\log \theta(mb)}\right)\right) + O\left(\frac{r_m x}{\log^22x}\right).
$$ 
\end{lemma}

\begin{proof}
As in the proof of Lemma \ref{lem-B(x)starter1}, we may assume $x\ge2$.
We use Lemma \ref{lem-B(x)simplification} to combine the first two terms in the expression in 
Lemma \ref{lem-B(x)starter1}, getting 
\begin{align*} 
B_m(x) = &\frac{6}{\pi^2} x \sum_{mb > e^{\sqrt{\log x}}} 
\frac{\chi_m(mb)}{mb} \prod_{p \leq \theta(mb)} \left(1 + \frac{1}{p}\right)^{-1} \\ 
&- x \sum_{e^{\sqrt{\log x}} < mb \leq \sqrt{x} } 
\frac{\chi_m(mb)}{mb} e^\gamma \omega\left(\frac{\log(x/mb)}{\log \theta(mb)}\right) 
\prod_{p \leq \theta(mb)} \left(1 - \frac{1}{p}\right) + O\left(\frac{r_m x}{\log^2 x}\right).
\end{align*} 
Note that $\frac{6}{\pi^2} = \prod_{p } \left(1 - \frac{1}{p^2}\right)$, so we have for any $n \geq 1$
$$
\frac{6}{\pi^2} \displaystyle\prod_{p \leq \theta(n)} \left(1 + \frac{1}{p}\right)^{-1} 
= \prod_{p \leq \theta(n)} \left(1 - \frac{1}{p}\right)\prod_{p>\theta(n)}\left(1-\frac1{p^2}\right)
=\prod_{p\le\theta(n)}\left(1-\frac1p\right)\cdot\left(1+O\left(\frac1n\right)\right).
$$ 
Thus, we have 
\begin{align*}
B_m(x) & =   x \sum_{mb > e^{\sqrt{\log x}} } \frac{\chi_m(mb)}{mb} \prod_{p \leq \theta(mb)} 
\left(1 - \frac{1}{p}\right) \\
& \quad - x\sum_{e^{\sqrt{\log x}} < mb \leq \sqrt{x} } \frac{\chi_m(mb)}{mb} 
e^\gamma \omega\left(\frac{\log x/mb}{\log \theta(mb)}\right) 
\prod_{p \leq \theta(mb)} \left(1 - \frac{1}{p}\right) + O\left(\frac{r_m x}{\log^2 x}\right) \\ 
& = x \sum_{mb \geq  e^{\sqrt{\log x}} } \frac{\chi_m(mb)}{mb} 
\left(1-e^\gamma\omega\left(\frac{\log(x/mb)}{\log\theta(mb)}\right)\right)
\prod_{p \leq \theta(mb)} 
\left(1 - \frac{1}{p}\right) 
+O\left(\frac{r_m x}{\log^2 x}\right),
\end{align*} 
using that $\omega(u)=0$ for $u<1$.
We use a strong form of Mertens' theorem (essentially, the prime number theorem)
to estimate the product over primes, which yields 
$$
\frac{e^{-\gamma}}{\log \theta(mb)}\left(1 + O\left(\frac{1}{\log^4 \theta(mb)}\right)\right).$$ 
Inserting this into our last expression for $B_m(x)$, we get
$$
B_m(x) = x \sum_{mb >e^{\sqrt{\log x}} } \frac{\chi_m(mb)}{mb} \frac{1}{\log \theta(mb)}
\left(e^{-\gamma}-\omega\left(\frac{\log(x/mb)}{\log \theta(mb)}\right)\right) + 
O\left(\frac{r_m x}{\log^2 x}\right).
$$ 
Using \eqref{eq:omega}, the sum may be extended to all $mb\ge1$ introducing an
acceptably small error, and so proving the lemma.
\end{proof}

The analogue of \cite[Lemma 5.7]{weingartner} is as follows.

\begin{lemma} 
\label{lem-5.6}
For $m\ge 1$, $x \geq 1$ we have 
$$
B_m(x) = x \int_1^\infty \frac{B_m(y)}{y^2 \log 2y} 
\left(e^{-\gamma} - \omega\left(\frac{\log(x/y)}{\log 2y}\right)\right) \mathrm{d}y + O\left(\frac{r_m x}{\log^22x}\right).
$$ 
\end{lemma}

\begin{proof}
First, replace the two instances of $\theta(mb)$ in Lemma \ref{lem-5.5}  by $2mb$. 
Second, use partial summation to replace $\chi_m$ by $B_m$. 
All new error terms introduced in these two steps turn out to be $\ll \frac{x \log 2m}{m (\log2x)^2}$.
We omit the details, since the calculations are identical to those in Lemmas 5.6 and 5.7 of \cite{weingartner}
\end{proof}

At this point we may conclude the proof of Theorem \ref{thm-starterasymp} by following the proof of \cite[Theorem 1.3]{weingartner},
replacing $t$ by $2$, $D(x,t)$ by $B_m(x)$, and $\alpha(t)$ by 
$$ \alpha_m = e^{-\gamma} \int_1^\infty \frac{B_m(y)}{y^2 \log 2y} dy \ll \frac{\log 2m}{m}.$$ 
All new error terms turn out to be $\ll r_m x /\log^2 2x$.


\section{Starters}\label{sec-starters}

When considering $\varphi$-practical numbers $n$ with a given squarefull part $s$, it is
natural to consider certain ``primitive" $\varphi$-practical numbers which
have squarefull part $s$, which we call {\em starters}.

\begin{definition}
A starter is a $\varphi$-practical number $m$ such that either $m/P^+(m)$ is not
$\varphi$-practical or $P^+(m)^2\mid m$.
A $\varphi$-practical number $n$ is said to have starter $m$
if $m$ is a starter, $m$ is an initial divisor of $n$, and $n/m$ is squarefree.  
\end{definition}

In some cases, it can be simple to characterize all of the starters with a given squarefull part. For example, $4$ is the only starter for $4$. Similarly, there are only three starters for $49$: $294=2\cdot3\cdot7^2$, $1470=2\cdot3\cdot5\cdot7^2$, and $735=3\cdot5\cdot7^2$. For other squarefull numbers, examining the corresponding set of starters becomes much more complicated. For example, there are infinitely many starters with squarefull part 9.

It is easy to see that each $\varphi$-practical number has a unique starter, so the
starters create a natural partition of the $\varphi$-practical numbers.
Note that in the notation of Section \ref{sec-freely}, with $\theta(x)=x+2$,
if $m$ is a starter, then $\mathcal{B}_m$
is the set of $\varphi$-practical numbers with starter $m$. 
Since we learned the asymptotics for each $B_m(x)$ in Theorem \ref{thm-starterasymp},
and since the sets $\mathcal{B}_m$, with $m$ running over starters, partition the
set of $\varphi$-practicals, it would seem that the proof of Theorem \ref{thm-main}
is now complete.  However, it remains to show that starters are so scarce that the sums 
$\sum c_m$ and $\sum r_m$ over starters $m$ are finite.  
We begin with the following corollary of Theorem \ref{thm-starterasymp}.
\begin{cor}
\label{cor-cmsum}
With the variable $m$ running over starters, we have $\sum c_m<\infty$.
\end{cor}
\begin{proof}
 Since we know from Theorem \ref{thm-thompson}
that the number of $\varphi$-practical numbers in $[1,x]$ is $O(x/\log x)$,
the corollary follows immediately from Theorem \ref{thm-starterasymp}
and the fact that the sets $\mathcal{B}_m$ are disjoint as $m$ varies over starters.
\end{proof}

Let 
$$
H(n)=\frac{n+1}{\varphi(n)}.
$$
For a starter $m>1$, let $\alpha(m)$ denote the largest proper initial divisor of $m$ that is
$\varphi$-practical.
For example, $\alpha(3^2\cdot5\cdot7)=1$ and $\alpha(3\cdot5^2\cdot7)=3$.
\begin{lemma}
\label{lem-starter}
If $m$ is a starter with squarefull part $s>1$, then $P^+(\alpha(m))< P^+(s)$.
\end{lemma}
\begin{proof}
Suppose $P^+(\alpha(m))\ge P^+(s)$.  Then $m/\alpha(m)$ is squarefree and $>1$.   Since
$m$ satisfies the weak $\varphi$-practical property and both $\alpha(m)$ and $m$
are $\varphi$-practical, we would have $m/P^+(m)$ being $\varphi$-practical,
contradicting the definition of a starter.
\end{proof}

\begin{thm}
\label{thm-startercond}
Let $m$ be a starter and write $m=ak$ where $a=\alpha(m)$.
Then $H(ak)\ge H(a)$ and if $d$ is an initial divisor of $k$ with $1<d<k$, then $H(ad)<H(a)$.
\end{thm}
This theorem will be proved in the next section.

\begin{cor}
\label{lem-starter2}
Let $m$ be a starter with squarefull part $s>1$, let $a=\alpha(m)$, and write $m=ak$.  Then $H(ak)>H(a)$,
$k/\varphi(k)\ge1+1/a$,
and if $d$ is an initial divisor of $k$, $d<k$, then $d/\varphi(d)<1+1/a$.
\end{cor}
\begin{proof}
Theorem \ref{thm-startercond} gives that $H(ak)\ge H(a)$.  Suppose that $H(ak)=H(a)$.
Then $ak+1=(a+1)\varphi(k)$, which implies that $k$ is squarefree. On the other hand, by Lemma \ref{lem-starter}, $P^+(a) \leq P^+(s)$. Since $m = ak$ and $s \mid m$, then $P^+(a)$ must divide $k$ and it cannot divide $a$, hence $P^+(s) \mid m$. Since $s$ is the squareful part of $m$ then we must have $P^+(s)^2\mid k$, which means that $k$ is not squarefree.  This contradiction shows that $H(ak)>H(a)$.

Suppose the integer $b$ is coprime to $a$.  The condition $H(ab)>H(a)$ is equivalent to the condition 
$(ab+1)/\varphi(b) > a+1$.
Since $a+1$ is an integer and $(ab+1)/\varphi(b)$ is a rational number with denominator in lowest
terms a divisor of $\varphi(b)$, it follows that $H(ab)>H(a)$ is equivalent to $ab/\varphi(b)\ge a+1$,
which is equivalent to $b/\varphi(b)\ge 1+1/a$.  We have shown that $H(ak)>H(a)$, so that
$k/\varphi(k)\ge1+1/a$.  Further, if $d$ is an initial divisor of $k$ with $1<d<k$, the condition $H(ad)<H(a)$ from
Theorem \ref{thm-startercond} implies
that $d/\varphi(d)<1+1/a$.  This also holds for $d=1$, so the corollary is proved.
\end{proof}

\begin{thm}
\label{thm-startercount}
The number of starters $m\le x$ is at most
$$ x\exp\Bigl\{-\Bigl(\frac{\sqrt{2}}4+o(1)\Bigr)\sqrt{\log x \log \log x}\Bigr\} \qquad (x\to \infty).$$
\end{thm}
\begin{proof}  
Let $x$ be large and let $L=\exp(\sqrt{\log x \log\log x})$.
By \cite{DeB} (see in particular, equation (1.6)), the number of integers $m\le x$ with $P^+(m)\le L^{\sqrt{2}}+1$
is at most $x/L^{\sqrt{2}/4+o(1)}$ as $x\to\infty$, so we may assume that $P^+(m)>L^{\sqrt{2}}+1$.
Since the number of squarefull numbers at most $t$ is $O(t^{1/2})$, the number
of integers $m\le x$ divisible by a squarefull number at least $L^{\sqrt{2}/2}$ is
$O(x/L^{\sqrt{2}/4})$, and so we may assume that the squarefull part of $m$ is
smaller than $L^{\sqrt{2}/2}$.  In particular, we may assume that $P^+(m)^2\nmid m$.
Denote the set of starters $m\le x$ which satisfy these properties by $\mathcal{S}(x)$.
For a $\varphi$-practical number $a$, let
$$
\mathcal{S}_{a}(x)=\{m>1: m\in\mathcal{S}(x),~ \alpha(m)=a\},\quad
S_{a}(x)=\#\mathcal{S}_a(x).
$$
If $m\in\mathcal{S}_a(x)$, then $q=a+2$ is a prime divisor of $m$, and in
fact $q^2\mid m$.  (If $r=P^-(m/a)<a+2$ and $r^j\|m$, then $ar^j$ is $\varphi$-practical.
If we then have $ar^j=m$, then $j=1$ and $m$ is not a starter, or $j>1$ and
$m\not\in\mathcal{S}(x)$.  And if $ar^j<m$, then $ar^j\mid\alpha(m)=a$,
again a contradiction.  So $P^-(m/a)=a+2=q$, and if $q^2\nmid m$, then
$aq$ is $\varphi$-practical, again leading to a contradiction.)
The number of starters $m\le x$ with $a>L^{\sqrt{2}/6}$ is therefore
\begin{equation*}
 \sum_{a>L^{\sqrt{2}/4}}S_a(x) \le \sum_{a>L^{\sqrt{2}/6}} \frac{x}{a(a+2)^2} \ll \frac{x}{L^{\sqrt{2}/3}}.
\end{equation*}
We will show that for all $\varphi$-practical numbers 
$a\le L^{\sqrt{2}/6}$, we uniformly have 
\begin{equation}\label{eq-Sa}
S_{a}(x)\le \frac{x}{a L^{\sqrt{2}/4+o(1)}},\qquad(x\to\infty).
\end{equation}
The desired result then follows from summing over $a$. 
It remains to establish \eqref{eq-Sa}, the proof of which is modeled after \cite{erdos}.

For a $\varphi$-practical number $a$, let $m\in\mathcal{S}_{a}(x)$ and 
write $m=ak$, where $P^+(a)<P^-(k)$.  
Since $m\in\mathcal{S}(x)$, it follows that $p=P^+(k)>L^{\sqrt{2}}+1$.
Moreover, we may assume that the squarefull part of $k$ is $< L^{\sqrt{2}/2}$. 
We write $k=pw$ and use the properties that
$k/\varphi(k)\ge1+1/a$ and $w/\varphi(w)<1+1/a$ (cf.\ Corollary \ref{lem-starter2}).
We have
\begin{equation}
\label{eq-squeeze}
1+\frac1a
\le \frac{k}{\varphi(k)}=\frac{w}{\varphi(w)}\left(1+\frac1{p-1}\right)
<\left(1+\frac1a\right)\left(1+\frac1{p-1}\right)<\left(1+\frac1a\right)\left(1+\frac1{L^{\sqrt{2}}}\right).
\end{equation}
We claim that $k$ has a divisor $d$ in $I:=[L^{\sqrt{2}/4},L^{\sqrt{2}/2}]$, with
$\gcd(d,k/d)=1$.  This holds if $k$ has a prime factor in $I$, since
it must appear just to the first power in $k$.  It also holds if the squarefull
part of $k$ is $\ge L^{\sqrt{2}/4}$.  So, assume that 
$k$ has no prime divisor in $I$ and that its squarefull part is smaller
than $L^{\sqrt{2}/4}$.  Write $k=uv$ where $P^+(u)< L^{\sqrt{2}/4}$ and $P^-(v)>L^{\sqrt{2}/2}$.  
Since $v$ has at most $\log x/\log L $ prime factors, it follows that
$$
\frac v{\varphi(v)}\le 1+\frac{\log x}{L^{\sqrt{2}/2}}
$$
for all sufficiently large $x$.  Now $p>L^{\sqrt{2}}> P^+(u)$, so $p$ does 
not divide $u$ and $u\mid w$.
Hence $u/\varphi(u)\le w/\varphi(w)<1+1/a$, which implies $u/\varphi(u)\le 1+1/a-1/(au)$. 
Also, $uv/\varphi(uv)\ge1+1/a$, which means that
$$
\left(1+\frac1a-\frac1{au}\right)\left(1+\frac{\log x}{L^{\sqrt{2}/2}}\right)\ge1+\frac1a,
$$
and so by our upper bound on $a$,
$$
u>\frac{L^{\sqrt{2}/2}+\log x}{(a+1)\log x}\ge L^{\sqrt{2}/4}.
$$
Since $P^+(u)< L^{\sqrt{2}/4}$ and the squarefull part of $u$ is $<L^{\sqrt{2}/4}$, 
it follows that $u$, and hence $k$, has a divisor $d$ in $I$ with
$\gcd(d,k/d)=1$, as claimed. 

For the elements $ak$ of $\mathcal{S}_{a}(x)$, we consider the
map which takes $k$ to $k/d$, where $d$ is the least divisor of $k$ in
$I$ that is coprime to $k/d$.
We claim this map is at most $L^{o(1)}$-to-one, as $x\to\infty$.  
For suppose $k\ne k'$ and $k/d=k'/d'$.
Then
\begin{equation}
\label{eq:rad}
\frac{k/\varphi(k)}{k'/\varphi(k')}=\frac{d/\varphi(d)}{d'/\varphi(d')}.
\end{equation}
If the right side of \eqref{eq:rad}
is not 1, assume without loss of generality that it is $>1$,
so that it is greater than 
$1+1/(dd')\ge 1+1/L^{\sqrt{2}}$.
But by \eqref{eq-squeeze}, the left side is less than $1+1/L^{\sqrt{2}}$.  
This contradiction shows that for a given pair $k,d$,
all pairs $k',d'$ that arise in our problem with $k/d=k'/d'$ 
have $d/\varphi(d)=d'/\varphi(d')$.  That is, $\rad(d')=\rad(d)$, where
$\rad(n)$ is the largest squarefree divisor of $n$.  It is shown in the
proof of \cite[Theorem 11]{ELP} (also see \cite[Lemma 4.2]{Poll}) that
the number of $d'\in I$ with this property is at most $L^{O(1/\log\log x)}$,
uniformly for $d\in I$, which proves our assertion about the map $k\mapsto k/d$.
Now $k/d\le x/(aL^{\sqrt{2}/4})$, which establishes 
\eqref{eq-Sa} and so completes the proof of the theorem.
\end{proof}

We are now ready to complete the proof of Theorem \ref{thm-main}.
By Theorem \ref{thm-starterasymp} and the discussion at the beginning of this section, we have, for $x\ge 2$,
$$P_\varphi(x) = \sum_{m\ge 1} B_m(x) =  \sum_{m\ge 1} \left(\frac{c_m x}{\log x}+ O\left(\frac{r_m x}{\log^2 x}\right)\right),$$
where $m$ runs over starters and $r_m=(\log 2m)^6 /m$. 
We have $C:= \sum_{m\ge 1} c_m < \infty$ by Corollary \ref{cor-cmsum},
while $\sum_{m\ge 1} r_m < \infty$ follows from Theorem \ref{thm-startercount} and partial summation.
Thus
$$ P_\varphi(x) = \frac{C x}{\log x} + O\left(\frac{x}{\log^2 x}\right),$$
where $C>0$, by Theorem \ref{thm-thompson}.

\section{Proof of Theorem \ref{thm-startercond}}

We begin with some notation and lemmas. For sets $A,B\subset\mathbb{R}$, and $\lambda\in \mathbb{R}$, let $A+B=\{a+b: a\in A, \ b\in B\}$,
$AB=\{ab:a\in A,\ b\in B\}$ and $\lambda A=\{\lambda\}A$.
For a nonnegative integer $n$, we let $[n]$ denote the set $\{0,1,\dots, n\}$.
\begin{lemma}
\label{lem-sumofsets}
Let $g,a,h$ denote nonnegative integers, with $a>0$.  Then
$[g]+h[a]=[g+h a]$ if and only if
\begin{equation}
\label{cond1}
h \le g+1.
\end{equation}
\end{lemma}
\begin{proof}
If $h\ge g+2$ then $g+1\notin[g]+h[a]$, so assume \eqref{cond1} holds.
The sumset consists of all of the integers in the intervals
$[ih,g+ih]$ for $i=0,1,\dots,a$.  The condition \eqref{cond1}
implies that for $i<a$, $(i+1)h\le g+ih+1$, so these intervals
cover every integer from 0 to $g+h a$.
\end{proof}

Let $S(n)$ denote the set of all sums of totients of distinct divisors of $n$, that is 
$S(n):=\left\{ \sum_{d|n} \phi(d) \varepsilon(d) : \varepsilon(d)\in \{0,1\} \right\}.$
Note that if $\gcd(m,n)=1$ then $S(mn)=S(m)S(n)$.  Also note that if $g,h,a$ are nonnegative
integers then for each positive integer $n$,
\begin{equation}
\label{largest}
S(n)=[g]+h[a]\quad\hbox{implies}\quad n=g+ha,
\end{equation}
as can be seen by examining the largest member of the two sets.

\begin{lemma}
\label{lem-apq}
Let $a$ be $\phi$-practical, $a+2=p<q\le a p+2$, $n=ap^\nu q^\mu$, where $\nu\ge 2$, $\mu\ge 1$. Then
$S(n) = [n-a\varphi(n/a)]+\varphi(n/a)[a]$.
\end{lemma}

\begin{proof}
We do a double induction, first on $\nu$, then on $\mu$.
We begin with the case $\nu=2$, $\mu=1$, $n=ap^2 q$.
We have 
$$S(n)=S(apq)+\phi(p^2)S(aq).$$ 
Since $apq$ is $\phi$-practical, $S(apq)=[apq]$,
while $S(aq)=[a]+(q-1)[a]$.  Since $\varphi(p^2)\le apq$, \eqref{cond1} is satisfied
in Lemma \ref{lem-sumofsets} with $g=apq,h=\varphi(p^2)$, so that
$$
S(n)=[apq+\varphi(p^2)a]+\varphi(p^2q)[a].
$$
Using \eqref{largest}, the result holds in this case.

Next, we consider $\nu\ge 2$, $\mu=1$, by induction on $\nu$. 
Using the induction hypothesis, write 
\begin{align*}
S(ap^{\nu+1}q)&=S(ap^\nu q)+ \varphi(p^{\nu+1})S(aq)\\
&=[ap^\nu q-a\varphi(p^\nu q)]+\varphi(p^\nu q)[a]
+\varphi(p^{\nu+1})[a]+\varphi(p^{\nu+1}q)[a].
\end{align*}
Since $\varphi(p^{\nu+1})\le ap^\nu q-a\varphi(p^\nu q)$ (note that it suffices to show it
for $\nu=1$, then use $a=p-2,q\ge5$), Lemma \ref{lem-sumofsets} implies we can
roll the first and third sets together, getting
$$
S(ap^{\nu+1}q)=[ap^\nu q-a\varphi(p^\nu q)+a\varphi(p^{\nu+1})]+\varphi(p^\nu q)[a]
+\varphi(p^{\nu+1}q)[a].
$$
We again apply Lemma  \ref{lem-sumofsets} to roll the first two sets together (it is easy
to show \eqref{cond1} holds), getting
$$
S(ap^{\nu+1}q)=[ap^\nu q+a\varphi(p^{\nu+1})]+ \varphi(p^{\nu+1}q)[a].
$$
Again using \eqref{largest}, we have the result for all $\nu\ge2$ and $\mu=1$.

Now we assume the result at $\nu,\mu$, where $\nu\ge2,\mu\ge1$ and prove it at
$\nu,\mu+1$.  We have, using the induction hypothesis, that
$$
S(ap^{\nu}q^{\mu+1})=S(ap^\nu q^\mu)+\varphi(q^{\mu+1})S(ap^\nu)
=[ap^\nu q^\mu-a\varphi(p^\nu q^\mu)]+\varphi(p^\nu q^\mu)[a]+\sum_{i=0}^\nu\varphi(p^i q^{\mu+1})[a],
$$
where the summation sign indicates a sum of sets.  We use Lemma \ref{lem-sumofsets} and roll
the various sets into the first set as far as possible.  We do this first separately for
$i=0,1,\dots,\nu-2$ (these can be done in any order since $\varphi(p^{\nu-2}q^{\mu+1})
\le ap^\nu q^\mu-a\varphi(p^\nu q^\mu)$), getting
$$
S(ap^\nu q^{\mu+1})=[ap^\nu q^\mu-a\varphi(p^\nu q^\mu)+ap^{\nu-2}\varphi(q^{\mu+1})]
+\varphi(p^\nu q^\mu)[a]+\varphi(p^{\nu-1}q^{\mu+1})[a]+\varphi(p^\nu q^{\mu+1})[a].
$$
We can now roll the second set into the first using Lemma \ref{lem-sumofsets}, and then
the next, each time easily verifying \eqref{cond1}.  We now have
$$
S(ap^\nu q^{\mu+1})=[ap^\nu q^\mu+a\varphi(p^{\nu-1}q^{\mu})]+\varphi(p^\nu q^{\mu+1})[a].
$$
By \eqref{largest}, the result holds for $\nu,\mu+1$.  This completes our argument.
\end{proof}

\begin{lemma}
\label{lem-npnu}
Assume $a\mid n$, $a<n$, $S(n)=[g]+h [a]$, $g=n-a\varphi(n/a)$, $h=\varphi(n/a)$, and $P^+(n)<p\le g+2\le h$. Then, for $\nu\ge 1$, 
$ S(n p^\nu) = [\tilde{g}]+\tilde{h} [a]$, where $\tilde{g}=n p^\nu-a\varphi(n p^\nu/a)$ and $\tilde{h}=\varphi(n p^\nu/a)$.
\end{lemma}

\begin{proof}
We show the result by induction on $\nu$. 
For $\nu=1$, consider that
$$
S(np)=S(n)+(p-1)S(n)=[g]+h[a]+(p-1)[g]+(p-1)h[a].
$$ 
Since $p\le g+2$, condition \eqref{cond1} holds and so Lemma \ref{lem-sumofsets} implies that
$$
S(np)=[pg]+h[a]+(p-1)h[a].
$$
We apply Lemma \ref{lem-sumofsets} again, noting that $h\le pg$ is equivalent to
$\varphi(n/a)\le n/(a+1/p)$ and this inequality follows from
$\varphi(n/a)\le(n/a)(1-1/P^+(n/a))<(n/a)(1-1/p)$.  A call to \eqref{largest}
completes the proof when $\nu=1$. 

Now assume that the result holds for the set $S(np^{\nu})$, for some $\nu\ge 1$, and write
\begin{align*}
S(n p^{\nu+1})&= S(np^{\nu}) + \varphi(p^{\nu+1}) S(n)\\
&=[np^\nu-a\varphi(np^\nu/a)]+\varphi(np^\nu/a)[a]+\varphi(p^{\nu+1})[g]+\varphi(p^{\nu+1})h[a].
\end{align*}  
Since $p\le g+2$, we have $\varphi(p^{\nu+1})<np^\nu-a\varphi(np^\nu/a)$, so that Lemma \ref{lem-sumofsets} gives
$$
S(np^{\nu+1})=[np^\nu-a\varphi(np^\nu/a)+\varphi(p^{\nu+1})g]+\varphi(np^\nu/a)[a]+\varphi(p^{\nu+1})h[a].
$$
We have seen that $\varphi(n/a)\le pg$, so that $\varphi(np^\nu/a)\le \varphi(p^{\nu+1})g$, so that
one final call to Lemma \ref{lem-sumofsets} gives
$$
S(np^{\nu+1})=[np^\nu-a\varphi(np^\nu/a)+\varphi(p^{\nu+1})g+\varphi(np^\nu/a)a]+\varphi(p^{\nu+1})h[a].
$$
Since $\varphi(p^{\nu+1})h=\varphi(np^{\nu+1}/a)$, the lemma follows from \eqref{largest}.
\end{proof}

We are now ready to prove Theorem \ref{thm-startercond}.

Let $m$ be a starter and write $m=ak=a p_1^{\nu_1} p_2^{\nu_2}\cdots p_j^{\nu_j}$, where $P^+(a)<p_1<p_2<\ldots <p_j$ and $a=\alpha(m)$.
If $j=1$, we have $\nu_1\ge 2$, or else $m$ is not a starter. Since $m$ is $\varphi$-practical and $\nu_1\ge 2$, 
we have $p_1 \le a+1$, which is equivalent to $H(m)\ge H(a)$. (One needs to use the fact that $\nu_1 \geq 2$ in order to get equivalence with $p_1 \leq a+1$, rather than $p_1 \leq a+2.$)

If $j\ge 2$, we have $p_1=a+2$ and $\nu_1\ge 2$, because otherwise $ap_1^{\nu_1}$ would be $\varphi$-practical,
contradicting the definition of $\alpha(m)$.
Thus $H(ap_1^{\nu_1})<H(a)$, as required. The set $S(a p_1^{\nu_1})$ does not contain the integer $ap_1+1$,
which implies that $p_2 \le ap_1+2$. 
We write $n_i=a p_1^{\nu_1} p_2^{\nu_2}\cdots p_i^{\nu_i}$, 
$g_i=n_i-a\varphi(n_i/a)$, $h_i=\varphi(n_i/a)$,
for $1\le i \le j$. 
Lemma \ref{lem-apq} shows that $n_2$ is $\varphi$-practical if and only if $g_2 \ge h_2-1$, 
which is equivalent to $H(n_2) \ge H(a)$.
Note that $n_i$ is not $\varphi$-practical for $1\le i < j$ by the definition of $\alpha(m)$.

If $j\ge 3$, we proceed by induction on $i$, $2\le i <j$.
Lemma \ref{lem-apq} shows that $S(n_2)=[g_2]+h_2 [a]$, $g_2 < h_2-1$ (i.e.,  $H(n_2) < H(a)$),  
since $n_{2}$ is not $\varphi$-practical, and $p_3\le g_2+2$, as $g_2+1 \notin S(n_2)$.
Now assume $S(n_i)=[g_i]+h_i [a]$,  $g_i < h_i-1$ and $p_{i+1}\le g_i+2$, for some $i\ge 2$.
If $i+1<j$, Lemma \ref{lem-npnu} shows that $S(n_{i+1})=[g_{i+1}]+h_{i+1} [a]$,  
$g_{i+1} < h_{i+1}-1$ (i.e.,  $H(n_{i+1}) < H(a)$), since $n_{i+1}$ is not $\varphi$-practical, 
and $p_{i+2}\le g_{i+1}+2$, since $g_{i+1}+1 \notin S(n_{i+1})$.
If $i+1=j$, Lemma  \ref{lem-npnu}  implies $S(n_j)=[g_j]+h_j [a]$ and  $g_j \ge h_j-1$ (i.e., $H(m) \ge H(a)$), since $n_j=m$ is $\varphi$-practical.

This completes the proof of Theorem \ref{thm-startercond}


\section{A heuristic estimate for the constant $C$ in Theorem \ref{thm-main}}

For the purpose of this heuristic, we ignore the error term in Lemma \ref{lem-5.6}, and change variables $v:=\frac{\log x}{\log 2}$,  
$u:=\frac{\log y}{\log 2}$ and $\beta_m(v):=B_m(x)/x $. 
The resulting integral equation for $\beta_m(v)$ matches the integral equation in Lemma 4 of \cite{IDD3} for the function $d(v)$, with the constant term $1$
replaced by some suitable constant. 
Corollary 6 of \cite{IDD3} gives an asymptotic formula for $(v+1)d(v)$ with a main term of the form
$$
R(v) =C +\frac{a}{(v+1)^{b}} +
2\mathrm{Re}\left(\frac{d+e i }{ (v+1)^{f+gi}}\right),$$
which corresponds to a Laplace transform (of $G(z):=e^z d(e^z-1)$) having a pole at the origin with residue $C$, 
a real pole at $-b$ with residue $a$, and two complex poles at $-(f\pm g i)$ with residues $d\pm e i$.
This suggests that $(v+1) \beta_m(v)$,  
and hence $\frac{P_\varphi(x)}{x/\log 2x}$, may also be well approximated by a function of this type.

\begin{figure}[h!]
\begin{center}
\includegraphics[height=9cm,width=14.6cm]{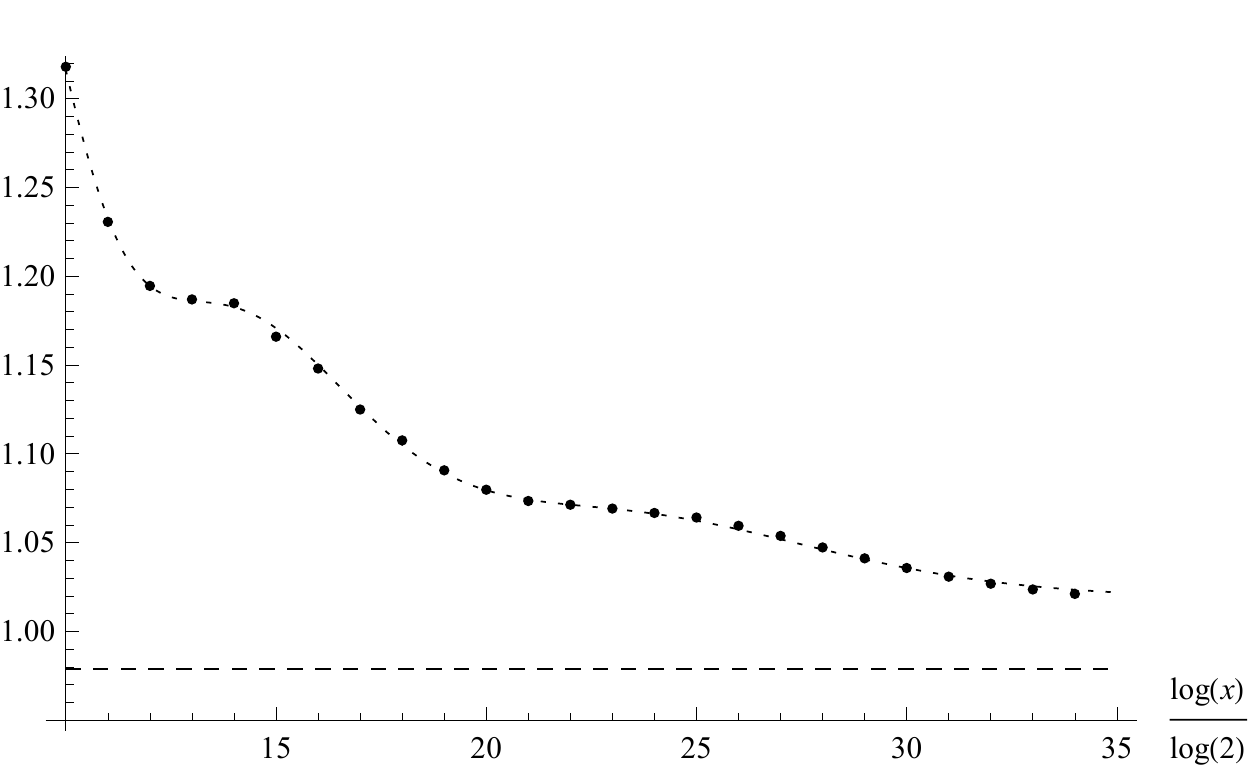}
\caption{The actual values of $P_\varphi(x)/(x/\log 2x)$ (dots) are well approximated by $R(v)$ (dotted line). Also shown is $\lim_{v\to \infty} R(v)$ (dashed line).}
\label{figure1}
\end{center}
\end{figure}

After calculating $P_\varphi(x)$ for $x=2^v$, $10\le v \le 34$, nonlinear regression leads to the model
$$
R(v) = 0.979154 +\frac{17.3307}{(v+1)^{1.66071}}+2\mathrm{Re}\left(\frac{-2.94536-4.82409 i}{ (v+1)^{2.30768+12.7422 i}}\right),
$$
which is shown with the actual values of $P_\varphi(x)/(x/\log 2x)$ in Figure \ref{figure1}.
Since $\lim_{v\to \infty} R(v) =0.979154$, our heuristic leads to the provisional estimate 
$$ \lim_{x\to \infty} \frac{P_\varphi(x)}{x/\log x} = \lim_{x\to \infty} \frac{P_\varphi(x)}{x/\log 2x} \approx 0.98 .$$

\section*{Acknowledgements}

The second author is supported by an AMS Simons Travel Grant. This work began while the second author was visiting Dartmouth College during the spring of 2015. She would like to thank the Dartmouth Mathematics Department for their hospitality. 

\bibliographystyle{amsplain}
\bibliography{practicalarxiv}

\end{document}